\documentclass[12pt]{amsart}
\usepackage{fullpage}
\usepackage{times}
\usepackage{latexsym}
\usepackage{amssymb}
\usepackage[all]{xy}

\newtheorem{thm}{Theorem}[section]

\newtheorem{lem}[thm]{Lemma}

\theoremstyle{remark}
\newtheorem{rem}[thm]{Remark}

\usepackage{hyperref}

\newcommand{\Z}{\mathbb{Z}}

\newcommand{\SO}{\mathrm{\SO}}

\title[Steenrod problem and the domination relation]
{Steenrod problem and the domination relation}

\author{Jean-Fran\c cois Lafont and Christoforos Neofytidis}
\address{Department of Mathematics, Ohio State University, Columbus, OH 43210, USA}
\email{jlafont@math.ohio-state.edu}
\address{Section de Math\'ematiques, Universit\'e de Gen\`eve, 2-4 rue du Li\`evre, Case postale 64, 1211 Gen\`eve 4, Switzerland}
\email{Christoforos.Neofytidis@unige.ch}
\date{\today}

\begin{document}

\begin{abstract}
We indicate how to combine 
some classical topology (Thom's work on the Steenrod problem) with some modern topology 
(simplicial volume) to show that every map between certain manifolds must have degree zero. 
We furthermore discuss a homotopy theoretic interpretation of parts of our proof, using Thom spaces 
and Steenrod powers.
\end{abstract}

\maketitle

\section{introduction}

Gromov introduced the {\it domination relation} on closed oriented $n$-dimensional manifolds~\cite[pg. 173]{CT}. Given two closed oriented $n$-dimensional 
manifolds $M$ and $N$, we say $M$ {\it dominates} $N$ if there is a map $f\colon M\rightarrow N$ of degree $\deg(f)\neq0$. 
In this short note, we add to the known
list of obstructions by showing the relevance of Thom's work~\cite{Thom} on the Steenrod problem~\cite{EilenbergSteenrod}.

\begin{thm}\label{t:main}
Let $M_0$, $N_0$, $M'$, and $N'$ be closed oriented $n$-dimensional manifolds, $7\leq k\leq n-3$ an integer, and having the following properties:
\begin{itemize}
\item[(a)] Assume $N_0$ has a $k$-dimensional homology class that is not representable by manifolds;
\item[(b)] Every $k$-dimensional homology class of $M_0$ is representable by manifolds;
\item[(c)] $N'$ has positive simplicial volume;
\item[(d)] $M'$ is a product of closed oriented  
manifolds of dimensions $\leq 9$, where all factors have positive simplicial volume and torsion-free
homology in degrees $\leq k$.
\end{itemize}
Then there are infinitely many pairs of positive integers $(s,t)$ with the property that $||M_0 \#_s M'|| \geq ||N_0 \#_t N'||$, but nevertheless
every continuous map $f\colon M_0 \#_s M' \rightarrow N_0 \#_t N'$ has $\deg(f)=0$.
\end{thm} 

In general, it is a difficult question to determine whether there is a non-zero degree map between two manifolds. Obstructions to the existence of a map of non-zero degree have been developed using a variety of tools from algebraic topology; we refer to~\cite{dH} for a recent survey on related results. 
One of the most basic methods comes from comparing the cohomology of $M$ and $N$. Some known obstructions arising from 
cohomology include: (i) inequalities on the ranks of the (co)homology groups, (ii) injectivity of induced homomorphisms in cohomology,
(iii) (sub)ring structures of the cohomology ring, and (iv) structure of the cohomology ring as a module over the Steenrod algebra; see~\cite{DW,Neo,LN,K} for various results that involve the aforementioned techniques. However, when the cohomology rings of the
manifolds are not {\it a priori} known, such as in Theorem \ref{t:main}, then these basic obstructions cannot be directly used. 

Another approach to obstructing a non-zero degree map comes from semi-norms on homology, such as the simplicial volume $||M||$. Recall that the simplicial
volume $||M||$ is a non-negative real number (introduced by Gromov~\cite{Gromov}) that roughly measures how efficiently $M$ can be ``triangulated over $\mathbb R$''. This
invariant has the so-called {\it functorial} property that, if $f:M\rightarrow N$ is a continuous map, then $||M||\geq|\deg(f)|\cdot||N||$. While the simplicial volume
cannot be used by itself to establish Theorem \ref{t:main}, it will feature prominently in the proof. 

In fact, the key idea of our proof is to use the
Steenrod problem as a bridge between these two main tools -- cohomology obstructions, and simplicial volume. This allows us to use {\it both}
techniques together in a situation where neither one individually is strong enough to work.

\subsection*{Outline}
In Section \ref{s:proof} we prove Theorem \ref{t:main} and in Section \ref{s:examples} we give examples of manifolds satisfying the assumptions of that theorem. Finally, in Section \ref{s:remarks} we discuss alterations of our hypotheses, as well as alternative arguments and approaches for parts of our proofs. 

\subsection*{Acknowledgments}
Part of this work was carried out during collaborative visits of the authors at Ohio State University and at University of Geneva. The authors thank these institutions for their hospitality. J.-F. L. was partially supported by the U.S.A. NSF, under grants DMS-1510640 and DMS-1812028. C. N. was partially supported by the Swiss NSF, under grant FNS200021$\_$169685.

\section{Proof of the Main Theorem}\label{s:proof}

In order to establish Theorem \ref{t:main}, we proceed by contradiction. Assuming that there is a non-zero degree map
$f\colon M_0 \#_s M' \longrightarrow N_0 \#_t N'$, we will see that this forces certain constraints on the integers $s$ and $t$, and
that there are infinitely many pairs of integers for which these constraints fail to hold. To simplify notation, let us denote
by 
$$M_s := M_0 \#_s M' \ \text{and} \ N_t:= N_0 \#_t N',$$ 
and assume we have a non-zero degree map $f\colon M_s\rightarrow N_t$.

We recall that a degree $k$ homology class $x\in H_k(X ; \mathbb Z)$ is {\it representable}
(by manifolds) if there exists a closed oriented $k$-manifold $Y$ and a map $\phi\colon Y\rightarrow X$ with the property that 
$\phi_*([Y]) = x$, where $[Y]\in H_k(Y; \mathbb Z)$ denotes the fundamental class of $Y$. Thom~\cite{Thom} proved that, when $X$ is an $n$-manifold, homology classes are always representable in degrees $1\leq k \leq 6$ and
degrees $n-2\leq k \leq n$. Thus one can only have non-representable classes in degrees $7\leq k \leq n-3$. In view of hypothesis
(a), we will henceforth focus on a degree $k$ within that range (in particular, $n\geq 10$). 
In the context of connected sums of manifolds $X= X_1 \# \cdots \# X_m$, Mayer-Vietoris gives us a splitting
$$H_k(X ; \mathbb Z) = \bigoplus _{i=1}^m H_k(X_i ; \mathbb Z).$$
Let us denote by $\rho_i \colon H_k(X_i ; \Z) \rightarrow H_k(X ; \Z)$ the injective homomorphism induced from the Mayer-Vietoris
sequence. We will need the following elementary result.

\begin{lem}\label{conn-sums}
For a connect sum $X= X_1 \# \cdots \# X_m$ of closed $n$-manifolds, and $7\leq k \leq n-3$, the following two statements
are equivalent:
\begin{enumerate}
\item every homology class $\alpha \in H_k(X ; \Z)$ is representable;
\item for every $i$, every homology class $\beta \in H_k(X_i ; \Z)$ is representable.
\end{enumerate}
\end{lem}

\begin{proof}
To see that (1) implies (2), let $\beta\in H_k(X_i ; \Z)$ be a homology class in one of the summands. Then via the Mayer-Vietoris
splitting, we can consider the homology class $(-1)^{i+1}\rho_i(\beta) \in H_k(X ; \Z)$. From (1), we can represent this homology class, so 
there is a map from a closed oriented $k$-manifold $\phi\colon Y^k \rightarrow X$ with $(-1)^{i+1}\rho_i(\beta) = \phi_*([Y])$. Let $q_i\colon X\rightarrow X_i$ 
be the map that collapses all the other summands to points. Then it is easy to see that the composite $(q_i)_* \circ (-1)^{i+1}\rho_i$ is the
identity map on $H_k(X_i ; \Z)$. So by taking the composite $q_i \circ \phi\colon Y\rightarrow X \rightarrow X_i$, we obtain a continuous
map from a closed oriented $k$-manifold having the property that 
$$(q_i \circ \phi)_* ([Y]) = (q_i)_*\left(\phi_*([Y])\right) = (q_i)_*\left((-1)^{i+1}\rho_i(\beta)\right) = \beta.$$
Since $\beta$ and $i$ were arbitrary, this establishes (2).

To see that (2) implies (1), let $\alpha \in H_k(X ; \Z)$ be an arbitrary homology class. Using the Mayer-Vietoris splitting, we can
write $\alpha = \sum_i \rho_i( \alpha_i)$, where each $\alpha_i \in H_k(X_i ; \Z)$. From (2), we have for each $i$ a closed oriented
$k$-manifold $Y_i$ and a map $\phi_i\colon Y_i \rightarrow X_i$ having the property that $(\phi_i)_*([Y_i]) = \alpha_i$. 
Let now $Y:=Y_1\#\cdots\# Y_m$ and define the following composite map
$$Y\stackrel{q}\longrightarrow Y_1\vee\cdots\vee Y_m\xrightarrow{\vee_{i=1}^m\phi_i}X_1\vee\cdots\vee X_m  \stackrel{\rho}\longrightarrow X,$$
where $q$ is the quotient map pinching to a point the essential sphere that defines the connected sum $Y$, $\vee_{i=1}^m\phi_i$ restricts to $\phi_i$ on each $Y_i$, and $\rho$ restricts to the inclusion of each $X_i$ in $X$. Then 
\begin{align*}
(\rho\circ(\vee_{i=1}^m)\phi_i\circ q)_*([Y]) & = (\rho\circ(\vee_{i=1}^m\phi_i))_*([Y_1\vee\cdots\vee Y_m])\\
& = \rho_*\left(\sum_{i=1}^m\phi_*([Y_i])\right)\\
& = \rho_*\left(\sum_{i=1}^m\alpha_i\right) = \sum_{i=1}^m \rho_i(\alpha_i) = \alpha,
\end{align*}
and this completes the proof.
\end{proof}

An immediate consequence of Lemma \ref{conn-sums} is that, in view of hypothesis (a), every $N_t$ has some non-representable degree $k$ 
homology class. We would now like to consider the corresponding question for the source manifolds $M_s$. There is the following elementary

\begin{lem}\label{products}
For a product $X= X_1\times \cdots \times X_m$ of closed manifolds 
and a given integer $k$, assume that for every
factor $X_i$ we have
\begin{enumerate}
\item every homology class $\alpha \in H_j(X_i ; \Z)$, $j\leq k$, is representable, and
\item the homology groups $H_j(X_i ; \Z)$ are torsion-free for $j\leq k$.
\end{enumerate}
Then every homology class $\alpha \in H_j(X; \Z)$ is representable.
\end{lem}

\begin{proof}
It is sufficient to show that an additive basis for the homology $H_j(X; \Z)$ ($j\leq k$) is representable. 
Note that from hypothesis (2), a simple induction shows there is no Tor term in the K\"unneth formula for $H_j(X; \Z)$.
Thus $H_j(X; \Z)$ is generated by cohomology classes of the form $\alpha_1\otimes \cdots \otimes \alpha_m$,
where $\alpha_i \in H_{j_i}(X_i ; \Z)$ and $j=\sum j_i$. Note that each $j_i \leq j \leq k$, so from hypothesis (1), we have 
closed oriented manifolds $Y_i$ with $\dim(Y_i) = j_i$, and a map $\phi_i\colon Y_i \rightarrow X_i$ satisfying $(\phi_i)_*([Y_i])=\alpha_i$.
Then forming the product map
$$\prod_{i=1}^m \phi_i \colon \prod_{i=1}^m Y_i \rightarrow \prod_{i=1}^m X_i = X,$$
 we have that
$$(\prod_{i=1}^m \phi_i )_*\left( \left[\prod_{i=1}^m Y_i\right] \right) = \bigotimes_{i=1}^m \left((\phi_i)_*([Y_i])\right) = \alpha_1 \otimes \ldots \otimes \alpha_m= \alpha.$$
This completes the proof of the lemma.
\end{proof}

\begin{rem}
As we shall see in Section \ref{ss:product}, Lemma \ref{products} does not hold if we remove the torsion-free assumption (2).
\end{rem}

Now from hypotheses (b) and (d), it follows that for all of the $M_s$, every degree $k$ homology class is representable.
Indeed, Thom showed that for manifolds of dimension $\leq 9$, every homology class is representable. Applying Lemma \ref{products}, 
we see  $M'$ has every homology class representable.  Lemma \ref{conn-sums} then
tells us that every degree $k$ homology class of $M_s$ is representable.

Next, we claim that if $f\colon M_s\rightarrow N_t$ has non-zero degree, then 
$|\deg (f)| \geq 2$. To see this, let us assume $|\deg(f)|=1$ and argue by contradiction. Since $f$ has degree one, we know
that $f_*\colon H_k(M_s ; \Z)\rightarrow H_k(N_t ; \Z)$ is surjective. So pick a class $y\in H_k(M_s ; \Z)$ with the property that 
$f_*(y)=x$, where $x$ is the non-representable class (recall that such class exists by assumption (a) and Lemma \ref{conn-sums}). Then since every class in $M_s$ is representable, there is a closed oriented
$k$-manifold $Y$ and a map $\phi: Y \rightarrow M_s$ with $\phi_*([Y]) = y$. Composing with $f$, we obtain a map 
$f\circ \phi\colon Y\rightarrow N_t$, and 
$$(f\circ \phi)_*([Y]) = f_*\left(\phi_* ([Y])\right) = f_*(y) = x,$$
which contradicts the fact that $x$ was non-representable. Thus we see that the Steenrod problem yields a {\it lower bound} on 
the degree of our hypothetical map.

Finally, to complete the proof, let us make use of the simplicial volume to obtain an incompatible inequality. Recall that, for manifolds of dimension 
$>2$, the simplicial volume is additive under connected sums. From our discussion above, if $f$ is a non-zero degree map, then
$|\deg(f)|\geq 2$. Noting that
our manifolds have dimension $n\geq 10$, we obtain
$$2\leq |\deg(f)| \leq \frac{||M_0 \#_s M'||}{||N_0 \#_t N'||} = \frac{||M_0|| + s ||M'||}{||N_0|| + t ||N'||}$$
Solving, we see that $t$ has to satisfy the linear upper bound
$$t \leq s\left( \frac{||M'||}{2 ||N'||} \right) + \left(\frac{||M_0|| - 2 ||N_0||}{2 ||N'||}\right).$$
Note that, by hypothesis (c), we have $||N'||>0$, so the expression above is indeed well defined. On the other hand,
the condition $||M_s|| \geq ||N_t||$ translates to the inequality
$$t\leq s\left( \frac{||M'||}{ ||N'||} \right) + \left(\frac{||M_0|| - ||N_0||}{ ||N'||}\right).$$
We conclude that, as long as the integer $t$ satisfies the inequality 
$$s\left( \frac{||M'||}{2 ||N'||} \right) + \left(\frac{||M_0|| - 2 ||N_0||}{2 ||N'||}\right)< t \leq s\left( \frac{||M'||}{ ||N'||} \right) + \left(\frac{||M_0|| - ||N_0||}{ ||N'||}\right)$$
there are no non-zero degree maps $f\colon M_s \rightarrow N_t$, even though $||M_s||\geq ||N_t||$. By hypothesis (d), the simplicial volume of the factors of $M'$ is positive, and since positivity of simplicial volume is inherited by products~\cite{Gromov}, we conclude 
that $||M'||>0$. Thus the linear bounds above both have (different) positive slopes. 
Comparing these slopes, it immediately follows that there are infinitely
many pairs of positive integers $s,t$ which satisfy the inequality above, concluding the proof of Theorem \ref{t:main}.

\section{Examples}\label{s:examples}

In this section we give examples of manifolds that fulfill the hypotheses of Theorem \ref{t:main}. 

\subsection{Examples of $N_0$, $N'$.}\label{ss:N_0,N'}
 Concerning $N_0$, Thom's work \cite{Thom} already featured examples of spaces with some non-representable 
homology class. For instance, he shows that one can take the product of two $7$-dimensional Lens spaces 
$L(7,3) \times L(7,3)$, resulting in a closed $14$-dimensional manifold containing an explicit $7$-dimensional
homology class which is not representable (see \cite[pgs. 62-63]{Thom}). 
Another explicit example can be found in a paper of Bohr--Hanke--Kotschick, who consider the $10$-dimensional
compact Lie group $Sp(2)$, and give an explicit $7$-dimensional homology class which is not representable (see \cite[pgs. 484-485]{BHK}).
Taking products with spheres yields higher dimensional examples with not representable classes. 

If the reader so desires, one can even arrange for $N_0$ to be aspherical. Indeed, one can apply the hyperbolization technique of 
Charney--Davis \cite{CD} to an $n$-manifold $X$ with a not representable homology class. The resulting $n$-manifold $h(X)$ will 
be aspherical, and it is easy to check that it will also have a not representable homology class. Iterating this, one can obtain 
infinitely many distinct $n$-manifolds, all of which have a not representable homology class. The cohomology rings of these 
manifolds are hard to explicitly compute.

Manifolds $N'$ are much easier to produce. Indeed, examples of manifolds with non-zero simplicial volume include
\begin{itemize}
\item closed aspherical manifolds with non-cyclic hyperbolic fundamental groups~\cite{Min,Min1,Gromov1}, such as closed negatively curved manifolds \cite{IY}, and
\item closed locally symmetric manifolds of non-compact type~\cite{BK,LS}.
\end{itemize}
Moreover, if $X$ is any $n$-manifold, then applying the 
Charney--Davis hyperbolization produces an $h(X)$ with the property that $||h(X)||>0$. Taking products and connected sums 
preserves the positivity of simplicial volume, giving rise to many possibilities for $N'$.

\subsection{Examples of $M_0$, $M'$.} Concerning $M_0$, we need to identify manifolds with all degree $k$ homology classes 
representable by manifolds. It is easier to seek manifolds with {\it all} homology classes (of any degree) representable by manifolds
-- call this property (R). There are a few basic examples and constructions that one can use. First
of all, tori always have property (R). Also, if follows from work of Thom \cite{Thom} that any closed manifold of dimension $\leq 9$ 
has property (R). 

In addition, one can use Lemma \ref{products} to produce higher dimensional examples with property (R). In order to do this, one needs
low dimensional examples which also have torsion-free homology groups (see our condition (d) for $M'$). In general, it seems difficult to
produce negatively curved manifolds with no torsion in their homology groups.   Examples are known in dimension $n=3$, where one
can use Dehn surgery to construct infinitely many closed hyperbolic $3$-manifolds which are (integral) homology $3$-spheres
(see \cite[pgs. 341-342]{RT}). This was extended to dimension $n=4$ by Ratcliffe--Tschantz \cite{RT}, who constructed infinitely
many aspherical homology $4$-spheres. Taking products of surfaces with these $3$- and $4$-dimensional examples, our Lemma
\ref{products} yields examples in all higher dimensions. The authors do not know of any non-product examples in dimensions
$\geq 5$.

Concerning $M'$, one wants to additionally ensure that the factors have positive simplicial volume. This is immediate for surfaces
and for the hyperbolic $3$-manifolds. While the $4$-dimensional examples of Ratcliffe--Tschantz are no longer hyperbolic, one can 
still show that the examples have positive simplicial volume. Thus any product of such manifolds would produce a valid $M'$.

\section{Concluding remarks}\label{s:remarks}

We finish our discussion with a few remarks on the hypotheses of our main result, as well as with a homotopy theoretic viewpoint of parts of its proof.

\subsection{On the product lemma.}\label{ss:product}
 We note that part of the difficulty in finding examples of $M_0$, $M'$ for our main theorem
is linked with the ``torsion-free homology'' hypothesis in our product Lemma \ref{products}. We note however that Lemma \ref{products}
is false if one removes hypothesis (2): Thom's example discussed above provides a counterexample. In Thom's example, $M:=L(7,3) \times
L(7,3)$ is a product of $7$-dimensional manifolds, so both factors have property (R). Nevertheless, Thom shows the product
$M$ does {\bf not} have property (R), as there is an explicit class in $H_7(M; \Z)$ which is non-representable. This is directly 
linked to the fact that, in the K\"unneth formula for the homology of $M$, the Tor term is non-zero (due to torsion in the homology
of $L(7,3)$). Thus the collection of homology classes that are ``obviously'' representable only form a finite index subgroup (of index
$9$) inside $H_7(M ; \Z)$. 

\subsection{The range of slopes}
Given any class $x\in H_k(N_t;\Z)$, Poincar\'e duality implies that, if $f\colon M_s\to N_t$ has non-zero degree, then there exists $y\in H_k(M_s;\Z)$ such that $f_*(y)=\deg(f)\cdot x$. Since every class in $M_s$ is representable by manifolds, we conclude as in the proof of Theorem \ref{t:main} that $\deg(f)\cdot x$ is representable. Thus, if the class $x$ in item (a) has the property that the multiples $x,  2\cdot x,...,(d-1)\cdot x$ are not representable, then this immediately yields the lower bound $|\deg(f)|\geq d$. Repeating the last calculation of the proof of our main theorem, which involves the simplicial volume, we conclude that as long as $t$ satisfies the inequality
$$s\left( \frac{||M'||}{d||N'||} \right) + \left(\frac{||M_0|| - d ||N_0||}{d ||N'||}\right)< t \leq s\left( \frac{||M'||}{ ||N'||} \right) + \left(\frac{||M_0|| - ||N_0||}{ ||N'||}\right),$$
there are no maps $f\colon M_s\rightarrow N_t$ of non-zero degree, even though $\|M_s\|\geq\|N_t\|$.
Therefore, the integer $d$ on the left-hand side of the above inequality which reflects the minimum number for which $d\cdot x$ is representable by manifolds, determines the range where we can find pairs $(s,t)$ satisfying Theorem \ref{t:main}.

\subsection{Using mapping degree sets and other semi-norms.} 
In hypotheses (c) and (d) of Theorem \ref{t:main}, the manifolds $N'$ and $M'$ have positive simplicial volume. For $M'$ this is due to positivity of simplicial volume being preserved under products. Also, both $M_s$ and $N_t$ have positive simplicial volume because of the additivity of the simplicial volume under taking connected sums (recall that their dimension is $n\geq10$). It would be interesting to replace $N'$ and $M'$ with manifolds whose simplicial volume is zero, but which have another non-vanishing semi-norm $\nu$ so that $\nu(M_s)\geq \nu(N_t)$ but $M_s$ does not dominate $N_t$. One major difficulty is that not many examples of $\nu$ seems to be known with well-understood behaviour under taking products or connected sums. Furthermore, in order to apply Theorem \ref{t:main}, one would need all assumptions of item (d) to be satisfied (or just find a single $M'$ with every degree $k$ homology class representable). It is natural to believe that (non virtually trivial) circle bundles over negatively curved manifolds, certain hypertori bundles over locally symmetric manifolds of non-compact type, and products of such spaces are good candidates for such new $N'$ and $M'$. Indeed, for many of those manifolds it is known that they can be dominated by another manifold with only finitely many different degrees, see~\cite{NeoHopfdegrees,DLSW,Neo}.

\subsection{Homotopy theoretic viewpoint.}  
We used the representable versus non-representable classes to argue that the induced map $f_*$ on 
integral homology could not be surjective, and hence gave the lower bound $|\deg (f)| \geq 2$. We now indicate 
how Thom's work on the Steenrod problem \cite{Thom} gives a homotopical viewpoint on that portion of our proof.

Thom showed that $x\in H_k(X ;\mathbb Z)$ is 
representable if and only if the Poincar\'e dual $PD(x) \in H^{n-k}(X ; \mathbb Z)$ can be realized as a pullback 
$PD(x) = \psi^*(\tau)$, where $\psi\colon X \rightarrow MSO(n-k)$ is a continuous map, and $\tau \in MSO(n-k)$ is the Thom class in the 
Thom space $MSO(n-k)$. Of course, from the homotopical viewpoint, we can also view the Thom class as a map 
$$\tau\colon MSO(n-k) \rightarrow K(\Z, n-k),$$ well-defined up to homotopy. Post-composing with $\tau$ then induces, for any space $X$, a map 
$$
\xymatrix{
[X, MSO(n-k)] \ar[r]^{\tau \circ -} & [X, K(\Z, n-k)]. \\
}
$$
The target space is just $H^{n-k}(X ; \Z)$, and stating that a manifold $X$ has every degree $k$ homology 
class representable is equivalent to stating that the map above is surjective, i.e. that every map $g\colon X \rightarrow K(\Z, n-k)$ factors,
up to homotopy, through $\tau\colon MSO(n-k) \rightarrow K(\Z, n-k)$, 
$$
\xymatrix{
 & MSO(n-k) \ar[d]^\tau \\
 X \ar@{..>}[ru]^{\tilde g} \ar[r]_-{g} & K(\Z, n-k)\\
}
$$

With this in hand, we can interpret the inequality $|\deg (f)|\geq 2$ in our proof homotopically as follows. If $|\deg(f)|=1$, then
the induced map $f^*$ on cohomology is injective. We can now consider the following commutative diagram
$$
\xymatrix{
[N_t, MSO(n-k)] \ar[d]_{\tau \circ -}&  \ar@{->>}[d]^{\tau \circ -}  \ar[l]_{PT}[M_s, MSO(n-k)]\\
[N_t, K(\Z, n-k)] \ar@{^{(}->}[r]^{-\circ f}& [M_s, K(\Z, n-k)]  \\
}
$$
Indeed, Thom's work, our hypotheses (b) and (d), and Lemmas \ref{conn-sums}, \ref{products} allows us to see that the right hand 
vertical map is surjective. On the other hand, our hypothesis (a) and Lemma \ref{conn-sums} implies that the left hand vertical map
is {\bf not} surjective. Finally, the existence of the top map, which is guaranteed from the Pontrjagin-Thom construction, along with
commutativity of the diagram, yields a contradiction. Thus, from this viewpoint, our argument is making use of whether or not maps
to the classifying space $K(\Z, n-k)$ can factor through $MSO(n-k)$. It is tempting to wonder if one could similarly obtain
lower bounds on the degree by considering whether or not maps factor through other canonical spaces.

\subsection{Relation with Steenrod powers.} In both of the examples of $N_0$ mentioned in Section \ref{ss:N_0,N'}, the non-representability of the
homology class is detected via a suitable Steenrod power. It is natural to ask whether our main theorem could also be established
purely by considering the cohomology structure as a module over the Steenrod algebra (by arguments similar to those in \cite{LN}).
This is closed related to the problem of understanding the homotopy type of the spaces $MSO(n-k)$, which generally seems to be
a difficult problem. 

In contrast, it is worth noting that the unoriented version of Steenrod's problem, corresponding to maps into $MO(n-k)$, is indeed
completely detectable in terms of Steenrod squares (see \cite[Section 6, pgs. 36-43]{Thom}). From the homotopy theoretic viewpoint,
the key difference is that $MO(r)$ has the $2r$-homotopy type of a {\it product} of Eilenberg-MacLane spaces, i.e. the first $2r$ 
stages of the Postnikov tower are just an iterated product. In contrast, the Postnikov tower for $MSO(n-k)$ has non-trivial fiber bundles
appearing at early stages. This opens up the possibility of having some higher cohomology operations appearing as obstructions 
(see \cite{K} for an illustration of these sorts of phenomena).

\bibliographystyle{amsplain}

\end{document}